\documentclass{amsart}

\usepackage{amsmath, amsfonts, amssymb}

\usepackage{refcheck}

\usepackage{xcolor}

\usepackage{geometry}
\usepackage[shortlabels]{enumitem}

\usepackage[noadjust]{cite}

\newtheorem{theorem*}{Theorem}

\newtheorem{theorem}{Theorem}

\newtheorem{corollary}[theorem]{Corollary}

\newtheorem{lemma}[theorem]{Lemma}

\newtheorem{proposition}[theorem]{Proposition}
\newtheorem{remark}[theorem]{Remark}

\newtheorem{prop}[theorem]{Proposition}

\begin{document}

\title{Leibniz algebras with an abelian subalgebra of codimension two}

\thanks{
This work was supported by the Spanish Government, Ministry of Universities grant `Margarita Salas', funded by the European Union - NextGenerationEU {and by Junta de Andalucía, Consejería de Universidad, Investigación e Innovación: ProyExcel\_00780 ``Operator theory: An interdisciplinary approach"}.
}

\author[Amir Fernández Ouaridi]{A. Fernández Ouaridi}
\address{Amir Fernández Ouaridi. \newline \indent University of Cádiz, Department of Mathematics, Puerto Real (Spain).}
\email{{\tt amir.fernandez.ouaridi@gmail.com}}

\author{D. A. Towers}
\address{David A. Towers
\newline \indent Dept. of Mathematics and Statistics,
Lancaster University, Lancaster LA1 4YF (England)}
 \email{d.towers@lancaster.ac.uk}


\thispagestyle{empty}

\begin{abstract}
A characterization of the finite-dimensional Leibniz algebras with an abelian  subalgebra of codimension two over a field $\mathbb{F}$ of characteristic $p\neq2$ is given. In short, a finite-dimensional Leibniz algebra of dimension $n$ with an abelian subalgebra of codimension two is solvable and contains an abelian ideal of codimension at most two or it is a direct sum of a Lie one-dimensional solvable extension of the Heisenberg algebra $\mathfrak{h}(\mathbb{F})$ and $\mathbb{F}^{n-4}$ or a direct sum of a  $3$-dimensional simple  Lie algebra and $\mathbb{F}^{n-3}$ or a Leibniz one-dimensional solvable extension of the algebra $\mathfrak{h}(\mathbb{F}) \oplus \mathbb{F}^{n-4}$.

\bigskip

{\it 2020MSC}: 17A32, 17B05, 17B20, 17B30.

{\it Keywords}: Leibniz algebra, Lie algebra, abelian subalgebra, abelian ideal.
\end{abstract}

\maketitle

\section{Introduction}

Abelian subalgebras of Lie and Leibniz algebras play a key role in their structure. In fact, they have been a matter of study for a long time. Characterizing arbitrary Lie and Leibniz algebras of finite dimension with an abelian subalgebra of a fixed codimension $k$ is a cumbersome problem which can barely be addressed when the codimension is small enough. For example, if $k$ is one it is known that the Lie or Leibniz algebra must be solvable and must contain an abelian ideal of codimension one, see \cite{CNT15, BC12, CT23}. Likewise, a characterization of the Lie algebras (and Poisson algebras) with an abelian subalgebra of codimension two was given in \cite{FNT24}. Obtaining a similar characterization for the case of Leibniz algebras with an abelian subalgebra of codimension two is an interesting question which we tackle in this manuscript.

In this context, it is natural to consider the invariants $\alpha$ and $\beta$, corresponding respectively to the dimensions of abelian subalgebras and ideals of maximal dimension of a given algebra.
The systematic study of these invariants for Lie and Leibniz algebras has been pursued since the work by Burde and Ceballos \cite{BC12}. Let us briefly recall the principal results in this sense first.
Let $L$ denote a Lie algebra (or a Leibniz algebra) of dimension $n$. Maximal subalgebras that are abelian of Leibniz algebras over an algebraically closed field have codimension one, see \cite{CT23}. For solvable Lie algebras over an algebraically closed field of characteristic zero, it is known that $\alpha(L) = \beta(L)$, see \cite{BC12}. The computation of $\alpha$ and $\beta$ for complex Lie and Leibniz algebras of small dimension was given in \cite{Ceballos, CT23}. Over an arbitrary field of characteristic $p\neq2$, if $L$ is a Leibniz algebra and $\alpha(L)=n-1$, then $\beta(L)=n-1$. Moreover, if $L$ is a supersolvable Lie algebra or a nilpotent Leibniz algebra such that $\alpha(L)=n-2$, then $\beta(L)=n-2$, see  \cite{CT23}. Furthermore, if $L$ is a nilpotent or a supersolvable Lie algebra such that $\alpha(L)=n-3$, then $\alpha(L)=n-3$ (for $p\neq 2$), see \cite{CT14, T22}. If $L$ is a nilpotent Lie algebra  and $\alpha(L)=n-4$, then $\beta(L)=n-4$ (for $p\neq 2, 3, 5$), see \cite{T22}. Also, some of the mentioned results obtained for Leibniz algebras were later extended to Leibniz superalgebras in \cite{BCFN24}.

\medskip

The main purpose of this paper is to prove the following characterization of the Leibniz algebras with an abelian subalgebra of codimension two.

\begin{theorem*}\label{mainthm}
    Let $L$ be a Leibniz algebra of dimension $n$ over an arbitrary field $\mathbb{F}$ of characteristic $p\neq 2$ with an abelian subalgebra of codimension two. Then $L$ is solvable and contains an abelian ideal of codimension $k\leq 2$ or we have one of the following three situations
    \begin{enumerate}    
        \item $L$ is a $3$-step solvable Lie algebra and $L \cong \mathfrak{c}(m)\oplus \mathbb{F}^{n-4}$, where $m\in \mathfrak{sl}_2(\mathbb{F})$  is irreducible.  Moreover, $L^2 = \mathfrak{h}(\mathbb{F})$ is the Heisenberg algebra and $C(L) \cong \mathbb{F}^{n-4} + L^{(3)}$
        is the unique abelian ideal of maximal dimension  $n-3$.

        \item $L$ is an almost simple Lie algebra and $L \cong  \mathfrak{d}(m)\oplus \mathbb{F}^{n-3}$, where $m\in \mathfrak{sl}_2(\mathbb{F})$. Moreover, $C(L) \cong  \mathbb{F}^{n-3}$ is the unique abelian ideal of maximal dimension  $n-3$.

        \item $L$ is a $3$-step solvable algebra and     
        $L \cong  \mathfrak{e}(\varphi, \vartheta, v, n)$ for certain $v\in  I_L$ and $\varphi \in \textrm{Der}_l(B)$ such that the induced map $\varphi^* = \vartheta^*:B/C(B)\rightarrow B/C(B)$ is irreducible, where $B:=\textrm{Nil}(L)\cong \mathfrak{h}(\mathbb{F}) \oplus \mathbb{F}^{n-4}$. Moreover, $C(B)$ is the unique abelian ideal of maximal dimension $n-3$.
    \end{enumerate}
    The definitions of the algebras $\mathfrak{c}, \mathfrak{d}, \mathfrak{e}$ can be found in section~\ref{sec2} and section~\ref{sec3}.
\end{theorem*}

In summary, a Leibniz algebra of dimension $n$ over a field of characteristic $p\neq2$ with an abelian subalgebra of codimension two has an abelian ideal of codimension at most two 
or it is a direct sum of a Lie one-dimensional solvable extension of the algebra $\mathfrak{h}(\mathbb{F})$ and $n-4$ copies of the field or a direct sum of a 
$3$-dimensional simple  Lie algebra and $n-3$ copies of the field or a Leibniz one-dimensional solvable extension of the algebra $\mathfrak{h}(\mathbb{F}) \oplus \mathbb{F}^{n-4}$. In particular, if the field is quadratically closed, the situations (1) and (2) are not possible. Note that our result generalizes \cite[Theorem 3.5]{CT23}, which considers solvable Leibniz algebras, to the general case.  The proof of this theorem is divided over the next two sections in the following way. Let $L$ be a Leibniz algebra with an abelian subalgebra $A$ of codimension two, then we have two possible situations, either $A$ is a maximal subalgebra, considered in section~\ref{sec2}, or it is not, considered in section~\ref{sec3}. A detailed examination of this two cases is given in this paper. Some of our arguments use linear algebra results that involve commuting linear operators, which can be consulted in \cite{ST68}.

\subsection{Notation}
 Let us fix some notation. Recall that a (left) Leibniz algebra is a vector space $L$ endowed with a bilinear multiplication $[\cdot, \cdot]:L\rightarrow L$ satisfying the Leibniz rule $[x,[y,z]]=[[x,y],z]+[y, [x,z]]$. For a introduction to Leibniz algebras, we refer to \cite{AOR20}. We use the term abelian subalgebra to refer to a zero subalgebra. For $x\in L$, we denote by $L_x:L\rightarrow L$ (resp. $R_x:L\rightarrow L$) the linear operator of left multiplication given by $L_x(y)=[x, y]$ (resp. $R_x(y)=[y, x]$). We denote by $I_L$ the vector space generated by the elements $[x, x]$ for $x\in L$, then $[I_L, L] = 0$. The space $I_L$ is an ideal and it is zero if and only if $L$ is a Lie algebra. 
 We denote by $C(L)$ the center of a Leibniz algebra $L$ and by $C_{L}(A)$ the centralizer of a subalgebra $A$ of $L$. The normalizer of a subalgebra $A$ will be denoted by $N(A)$. The radical of $L$ is denoted by $R(L)$, the nilradical of $L$ is denoted by $\textrm{Nil}(L)$ and  the space $\left\{x\in L: [x, L] = 0\right\}$ is denoted by $\textrm{Ann$_{\ell}$}(L)$. The derived series is denoted $L^{(k+1)} := [L^{(k)}, L^{(k)}]$ with $L^{(0)} := L$. Finally, we denote by $\alpha(L)$ and $\beta(L)$ the dimension of an abelian subalgebra and ideal of maximal dimension in $L$, respectively.

\section{Abelian subalgebras which are maximal}
\label{sec2}

Recall that by \cite[Proposition 2.2]{CT23}, if the base field is algebraically closed,  then any abelian subalgebra which is a maximal subalgebra has codimension one. Thus, we may assume the field is not algebraically closed in this section. Moreover, recall that in the case when a Leibniz algebra contains an abelian subalgebra of codimension one, then it contains an abelian ideal of codimension one too, as it was proved in \cite[Theorem 2.1]{CT23}, assuming the field has characteristic $p\neq2$. 

\subsection{General results}
    Let $L$ be a Leibniz algebra of dimension $n$ with a subalgebra $A$. We denote by $M(A)$ the vector space generated by the maps $L_a$ with $a\in A$. Note that $M(A)$ is a subalgebra of the Lie algebra $\mathfrak{gl}_n(\mathbb{F})$. The following useful lemma is a consequence of the Fitting decomposition.

\begin{lemma}\label{lem1}
    Let $L$ be a non-abelian Leibniz algebra of dimension $n$. If $A$ is an abelian subalgebra of codimension $k>1$ which is maximal, then 
    there is a subspace $L_1\subsetneq L$ such that $L=A\oplus L_1$ and $[A, L_1]= L_1$.
\end{lemma}
\begin{proof}
     By the Leibniz rule, for $x\in L$ and $a, b\in A$, we have $[a,[b,x]]=[[a,b],x]+[b, [a,x]] = [b, [a,x]]$, so the maps $L_a$ and $L_b$ commute, that is, $L_aL_b = L_b L_a$. Consider the Fitting decomposition of $L$, with respect to the maps $M(A)$.
    Write $L=L_0 \oplus L_1$. Clearly, $A\subseteq L_0$. Since $M(A)$ acts nilpotently in $L_0$, there is some $x\in L_0$ such that $x\not\in A$ and $[A, x] \subseteq A$. 
    Suppose $I_L\subset A$, then $[a, x] + [x, a] \in A$ and $[a, x], [x, x]\in A$, for $a\in A$. Hence, the space $A+\mathbb{F}x$ is a subalgebra containing $A$, which is a contradiction. So assume $I_L\not\subset A$, then $L=A+ I_L$ and $[L, A]=0$. Now if $[x, x] \in A$, then $A+\mathbb{F}x$ is a subalgebra containing $A$, and if $[x, x] \not\in A$, then $A+\mathbb{F}[x, x]$ is a subalgebra containing $A$. In both cases, we obtain a contradiction. Therefore, $L_0 = A$ and it follows that $L_1$ is $M(A)$-invariant and $[A, L_1]=L_1$.
\end{proof}

{
\begin{prop}
    Let $L$ be a Leibniz algebra of dimension $n$ over an arbitrary field. Let $A$ be an abelian subalgebra of maximal dimension $n-m$ which is a maximal subalgebra. Then $\textrm{dim}(\textrm{Ann$_{\ell}$}(L))\geq n-m - (\lfloor m^2/4 \rfloor + 1)$.
\end{prop}
\begin{proof}
    Since $A$ is an abelian subalgebra of maximal dimension, we have $C(L)\subset A$. By Lemma~\ref{lem1}, there is a subspace $L_1\subsetneq L$ of dimension $m$ such that $L=A\oplus L_1$ and $[A, L_1]= L_1$.   
     Put $B=L_1$ and define $\theta : A \rightarrow \mathfrak{gl}(B) $ such that $\theta(a) = L_a|_B$. Then $\theta$ is a homomorphism from $A$ to $\mathfrak{gl}(B)$ with kernel $\textrm{Ann$_A^{\ell}$}(B)$. Hence $A/\textrm{Ann$_A^{\ell}$}(B) \cong D$, where $D$ is an abelian subalgebra of $\mathfrak{gl}(B)$. Hence $\dim (A/\textrm{Ann$_A^{\ell}$}(B))\leq \lfloor\frac{m^2}{4}\rfloor+1$. It follows that $\dim (\textrm{Ann$_A^{\ell}$}(B)) \geq n-m-(\lfloor\frac{m^2}{4}\rfloor+1)$. But $\textrm{Ann$_A^{\ell}$}(B)=\textrm{Ann$_A^{\ell}$}(L)\subseteq \textrm{Ann$_{\ell}$}(L)$, whence the result. 
\end{proof}}

\subsection{Abelian subalgebras of codimension two which are maximal subalgebras}

{\label{sec22}}

Consider the algebra $\mathfrak{a}(\lambda, \mu)$, where the parameters $\lambda = (\lambda_{ij}), \mu = (\mu_{ij}) \in \mathfrak{gl}_2(\mathbb{F})$, with basis $a, b , x, y$ given by the non-zero products
$$[a, x] = \lambda_{11} x + \lambda_{12} y, \quad     [a, y] = \lambda_{21} x + \lambda_{22} y, \quad [b, x] = \mu_{11} x + \mu_{12} y, \quad     [b, y] = \mu_{21} x + \mu_{22} y.$$ 
Note that the brackets are not assumed skew-symmetric. Clearly, this algebra is $2$-step solvable.

\begin{lemma}\label{alga}
    The algebra $\mathfrak{a}(\lambda, \mu)$ is a Leibniz algebra if and only if $\lambda\mu = \mu\lambda$. 
\end{lemma}
 \begin{proof}
         Since $[a, [b, z]] = [b, [a, z]]$ for $z\in \mathfrak{a}(\lambda, \mu)$, we have $\lambda\mu = \mu\lambda$. The converse follows by an straightforward verification of the Leibniz identity.
 \end{proof}

 \begin{lemma}\label{alga2}
    If $\textrm{span}(\lambda, \mu) = \textrm{span}(\lambda', \mu')$, then $\mathfrak{a}(\lambda, \mu) \cong \mathfrak{a}(\lambda', \mu')$. 
\end{lemma}
\begin{proof}
    Write $\lambda' =  \alpha_{11} \lambda + \alpha_{12} \mu$ and $\mu' = \alpha_{21} \lambda + \alpha_{22} \mu$. If $\textrm{dim}(\textrm{span}(\lambda, \mu)) = 0$, then both algebras are abelian. Now, if $\textrm{dim}(\textrm{span}(\lambda, \mu)) = 1$, then both algebras have a one-dimensional center and we can assume $\mu = \mu' = 0$. Choose the map sending $a$ to $\alpha_{11}^{-1} a$ and fixing $x$, $y$ and $b$.
    Finally, if $\textrm{dim}(\textrm{span}(\lambda, \mu)) = 2$, then $d := \textrm{det}(\alpha_{ij})\neq0$ and
    the map fixing $x$, $y$ and sending $a$ to $\alpha_{22} d^{-1} a - \alpha_{12} d^{-1}b$ and $b$ to $-\alpha_{21} d^{-1}a + \alpha_{11}d^{-1}b$
    is the desired isomorphism.
\end{proof}

Given $m \in \mathfrak{gl}_n(\mathbb{F})$, we denote by $\chi_m(t)$ the characteristic polynomial of $m$.

\begin{theorem}\label{thmmax}
    Let $L$ be a Leibniz non-Lie algebra of dimension $n$ over {an arbitrary field $\mathbb{F}$} of characteristic $p\neq 2$. Suppose $\alpha(L) = n-2$. If
    $A$ is an abelian subalgebra of codimension two which is a maximal subalgebra of $L$. Then $L$ is a $2$-step solvable algebra. Precisely, we have $L\cong \mathfrak{a}(id, m)\oplus \mathbb{F}^{n-4}$ where $m\in \mathfrak{gl}_2(\mathbb{F})$ with $\chi_m(t)$ irreducible. Also, $\beta(L)=n-2$, $C(L) = \mathbb{F}^{n-4}$ and $C(L) + L^2$ is an abelian ideal of maximal dimension.
\end{theorem}
\begin{proof}
    Let $A$ be an abelian subalgebra of codimension two which is a maximal subalgebra of $L$. By Lemma \ref{lem1}, we have $L=A\oplus L_1$, where $L_1\subset L$ is a two-dimensional vector space such that $[A, L_1]=L_1$. Let $x, y$ be a basis of $L_1$. We distinguish two cases depending on the ideal $I_L$.

    \medskip 
    
    \noindent\underline{Case $I_L\not\subseteq A$.} Then we have $L=A+I_L$, by the maximality of $A$. It follows that $[L, L] = [A+I_L, A+I_L] = [A, I_L]\subset I_L$. Since $I_L$ is abelian, $L$ is $2$-step solvable Leibniz non-Lie algebra. Also, we have $[A, L_1] = L_1$, then $[L_1, L_1] = 0$. Hence, $C(L)+L_1$ is an abelian ideal of $L$, because $[L, A] = [A+I_L, A] = 0$ and $[L_1, A] = 0$.
    Also, we have that the dimension of $C(L)$ is $n-4$, because the dimension of the abelian subalgebra $C(L)+L_1$ is at most $n-2$ and the dimension of $M(A)$ is at most two, because $\alpha(\mathfrak{gl}_2(\mathbb{F}))=2$.
    
    Let $a, b$ be a basis of the linear complement of $C(L)$ in $A$. Then, the algebra $L$ is determined by the products
    $[a, x],  [a, y], [b, x], [b, y]\in L_1$. Therefore, we have $L=\mathfrak{a}(\lambda, \mu)+ \mathbb{F}^{n-4}$, for certain linear independent $\lambda, \mu \in \mathfrak{gl}_2(\mathbb{F})$. By Lemma \ref{alga}, the matrices $\lambda , \mu$ commute. By the maximality of $A$, the set $\left\{\lambda, \mu\right\}$ should be irreducible in $L_1$, but also since this set of generators is maximal, then the algebra spanned by it contains the identity matrix. Hence, there is $m\in \mathfrak{gl}_2(\mathbb{F})$ such that $\lambda, \mu \in \textrm{span}(id, m)$. Moreover, the polynomial $\chi_m(t)$ must be irreducible. By Lemma \ref{alga2}, the assertion in the theorem follows.

    \medskip 

    \noindent\underline{Case $I_L\subseteq A$.}  Assume $I_L\neq0$, otherwise $L$ is a Lie algebra, and $I_L\neq A$, otherwise $A$ is not a maximal subalgebra. 
    Suppose $[x, y] \not \in A$. Then $A' := A+\mathbb{F}[x, y]$ is a subalgebra containing $A$. Indeed, note that $[y, x] \in A'$ because $I_L\subseteq A$. It follows $[L_1, L_1]\subset A'$. Observe that for any $a\in A$, we have $[a, [x, y]] = [[a, x], y] + [x, [a, y]] \in [L_1, L_1]$. Also, we have   
    $[[x, y], a] = [x, [y, a]] - [y, [x, a]]$. But $[x, [y, a] + [a, y]], [y, [x, a] + [a, x]] \in I_L \subset A$ and $[x, [a, y]], [y, [a, x]] \in A'$. Hence, $[[x, y], a]\in A'$. Moreover, we have $[[x, y], [x, y]]\in I_L \subset A$. Consequently, the space $A'$ is a subalgebra, which is a contradiction. 
    
    Next, suppose $[x, y] \in A$. Then we have $[y, x]\in A$ and $[L_1, L_1]\subset A$.  Observe that if $v\in L_1$ with $v\neq 0$ is invariant in $M(A)$, then $A + \mathbb{F}v$ is a subalgebra, because
    $[v, v] \in A$ and 
    $[v, a] = [v, a] - [a, v] + [a, v] \in I_L + \mathbb{F}v \subset A + \mathbb{F}v$. 
    Also, $\textrm{dim}(M(A))\leq 2$, because the maps $M(A)$ commute. Now, distinguish the following cases.
    \begin{itemize}[-]

    \item Case $\textrm{dim}(M(A)) = 2$. Since $M(A)$ is an abelian subalgebra of $\mathfrak{gl}_2(\mathbb{F})$ of maximal dimension, then there is some $b\in A$ such that $L_b=id$. It follows, $[x, a] = [[b, x], a] = [b,[x,a]] = -[b, [a, x]] = -[a, x]$ for any $a\in A$.  Similarly, $[y, a] = - [a, y]$.  
    Moreover, we have the relations
    $$[x, x] = [[b, x], x] = -[x, [b, x]]=-[x,x], \quad [x, y] = [[b, x], y] = -[x, [b, y]]=-[x,y], $$
    $$[y, x] = [[b, y], x] = -[y, [b, x]]=-[y,x], \quad [y, y] = [[b, y], y] = -[y, [b, y]]=-[y,y]. $$
    If $p\neq 2$, then $L$ is a Lie algebra, a contradiction.

        \item Case $\textrm{dim}(M(A)) = 1$. Let $h\in A$ such that $L_h\neq0$ is irreducible in $L_1$. Clearly, we have $h\not\in I_L$ and $L_h(L_1) = L_1$. Now, for $z\in L_1$ and $a\in A$ we have
        $[[h, z], a] = [h,[z,a]] = -[h, [a, z]] = -[a, [h, z]]$. Since $x, y\in L_h(L_1)$, we obtain $[z, a] = - [a, z]$ for all $z\in L_1$ and $a\in A$. 
        Also, for $z, z'\in L_1$ we have the equation
        $[[h, z], z'] = -[z, [h, z']] = [z, [z', h]] = [z', [z, h]] = -[z', [h, z]]$. Again, we obtain $[z, z'] = - [z', z]$ for all $z, z'\in L_1$. Therefore, we conclude that $L$ is a Lie algebra, since $p\neq 2$.
        
        \item The case $\textrm{dim}(M(A))= 0$ leads to a contradiction as explained above, because $\mathbb{F}x$ is $M(A)$-invariant.

    \end{itemize}
   Hence, we have shown that the only possibility is the algebra in the statement, proving the result. 
\end{proof}

\begin{remark}
     Consider the algebra $\mathfrak{a}(id, m)$ over $\mathbb{R}$, where $m = (m_{ij})$, $m_{11} = m_{22}=0$ and $m_{12} = -m_{21} = 1$. Then the non-zero products are given by
     $[a, x] =  x,     [a, y] = y, [b, x] = y,      [b, y] = - x.$
     This is an example of a Leibniz non-Lie algebra with an abelian subalgebra of codimension two $A=\textrm{span}(a, b)$ which is a maximal subalgebra.
\end{remark}

Let us introduce the following algebras for $\lambda = (\lambda_{ij}), \mu = (\mu_{ij}) \in \mathfrak{gl}_2(\mathbb{F})$. The algebra $\mathfrak{b}(\lambda, \mu)$ with basis $a, b , x, y$ is given by the next non-zero products
$$[a, x] = -[x, a] = \lambda_{11} x + \lambda_{12} y,    \quad  [a, y] = -[y, a] = \lambda_{21} x + \lambda_{22} y,  $$
$$[b, x] = - [x, b] = \mu_{11} x + \mu_{12} y, \quad      [b, y] = -[y, b] = \mu_{21} x + \mu_{22} y.$$ 
Also, the algebra $\mathfrak{c}(\lambda)$ with basis $a, b , x, y$ is given by
$$[a, x] = -[x, a] = \lambda_{11} x + \lambda_{12} y, \quad     [a, y] = -[y, a] = \lambda_{21} x + \lambda_{22} y, \quad [x, y] = -[y, x] = b.$$ 

\begin{remark}
    The algebra $\mathfrak{b}(\lambda, \mu)$ 
    is $2$-step solvable. Moreover, it is a Lie algebra if and only if $\lambda \mu = \mu \lambda$. Likewise, the algebra $\mathfrak{c}(\lambda)$ is $3$-step solvable. Moreover, it is a Lie algebra if and only if $\lambda\in \mathfrak{sl}_2(\mathbb{F})$. 
\end{remark}

The next proposition complements the result for Lie algebras obtained in \cite[Theorem 4.3]{FNT24}  with further observations for the solvable case. Recall that the algebra $\mathfrak{q}_3(\lambda)$, 
 introduced in \cite{FNT24}, where $\lambda = (\lambda_{ij}) \in M_2(\mathbb{F})$, is the vector space with basis $h, x, y$ and {skew-symmetric} multiplication given by
$$[h, x] = \lambda_{11} x + \lambda_{12} y, \quad     [h, y] = \lambda_{21} x + \lambda_{22} y, \quad [x, y]=h.$$ 

\begin{prop}\label{maxlie}
    Let $L$ be a Lie algebra of dimension $n$ over a field $\mathbb{F}$ of characteristic $p\neq2$ with $\alpha(L) = n-2$. If
    $A$ is an abelian subalgebra of codimension two which is a maximal subalgebra. Then one of the following occurs:   
    \begin{enumerate}
        \item  $L$ is a $2$-step solvable algebra and $L \cong \mathfrak{b}(id, m)\oplus \mathbb{F}^{n-4}$, where $m\in \mathfrak{gl}_2(\mathbb{F})$ with $\chi_m(t)$ irreducible.  Moreover, $\beta(L)=n-2$, $C(L) \cong  \mathbb{F}^{n-4}$ and $C(L) + L^2$ is an abelian ideal of maximal dimension.

        \item $L$ is a $3$-step solvable algebra and $L \cong \mathfrak{c}(m)\oplus \mathbb{F}^{n-4}$, where $m\in \mathfrak{sl}_2(\mathbb{F})$ with $\chi_m(t)$ irreducible.  Moreover, $\beta(L)=n-3$, $L^2$ is the Heisenberg algebra and $C(L) \cong \mathbb{F}^{n-4} + L^{(3)}$
        is an abelian ideal of maximal dimension.

        \item $L$ is almost simple and $L \cong  \mathfrak{d}(m)\oplus \mathbb{F}^{n-3}$, where $m\in \mathfrak{sl}_2(\mathbb{F})$ with $\chi_m(t)$ irreducible and $\mathfrak{d}(m) := \mathfrak{q}_3(m)$. Moreover, $\beta(L)=n-3$ and $C(L) \cong  \mathbb{F}^{n-3}$ is an abelian ideal of maximal dimension.
        \end{enumerate}
\end{prop}
\begin{proof}
    Let $L = A + L_1$ be the decomposition of $L$ with respect to the maps in $M(A)$. Let $x, y$ be a basis of $L_1$.
    There are two solvable cases in \cite[Theorem 4.3]{FNT24}. Let us denote $ad_z := L_z$ for $z\in L$ in this proof.

        \medskip

    \noindent\underline{The case when $L$ is $2$-step solvable.}
    This case arise when $[x,  y] = 0$. Then $\textrm{dim}(ad_{A}) = 2$ and $C(L) + L_1$ is an abelian ideal of codimension two. Suppose $ad_A$ is spanned by $ad_a$ and $ad_b$, then $a, b$ are linearly independent. It follows that $L = \mathfrak{b}(\lambda, \mu) \oplus \mathbb{F}^{n-4}$ for certain commuting linearly independent $\lambda, \mu  \in \mathfrak{gl}_2(\mathbb{F})$. By a similar argument used in the previous theorem, we have that $L\cong \mathfrak{b}(id, m) \oplus \mathbb{F}^{n-4}$ where $m\in \mathfrak{gl}_2(\mathbb{F})$ and the polynomial $\chi_m(t)$ is irreducible. The statement (1) in the theorem is obtained.

        \medskip

    \noindent\underline{The case when $L$ is $3$-step solvable.} This case arise when $b: = [x, y]\neq0$, $b\in A$ and $ad_b = 0$. Then $ad_A$ can be realized as a abelian subalgebra of $\mathfrak{sl_2(\mathbb{F})}$, because $0=[a, [x, y]] = [[a,x], y] + [x, [a, y]]$, and $\textrm{dim}(ad_{A}) = 1$. Therefore, we have that $L = \mathfrak{c}(m) \oplus \mathbb{F}^{n-4}$ for some $m \in \mathfrak{sl}_2(\mathbb{F})$ where the polynomial $\chi_m(t)$ is irreducible.   
    The assertion (2) in the theorem follows.
\end{proof}

In conclusion, the characterization of the Leibniz algebras $L$ over a field of characteristic $p\neq 2$ with an abelian subalgebra of codimension two which is a maximal subalgebra and such that $\alpha(L)=n-2$ is given by Theorem \ref{thmmax} and Proposition \ref{maxlie}.

\section{Abelian subalgebras of codimension two which are not maximal}
\label{sec3}

In this section we study the case in which the Leibniz algebra contains an abelian subalgebra $A$ of codimension two and a subalgebra $B$ of codimension one such that $A\subset B$. The case in which $A$ contains the distinguished ideal $I_L$ is considered first in the following lemmas.

\begin{lemma}\label{lemia2}
    Let $L$ be a Leibniz algebra of dimension $n$ over an arbitrary field $\mathbb{F}$ {of characteristic $p\neq2$ such that $\alpha(L) = n-2$}. Let $A$ be an abelian subalgebra of codimension two such that $I_L\subseteq A$. Suppose $B = A + \mathbb{F}e_1$ is a subalgebra of $L$ and suppose $A$ is an ideal of $B$. Moreover, suppose $M(A)$ does not act nilpotently in $L$. Then $A$ is not an ideal of $L$, but $\beta(L) = n-2$.
\end{lemma}
\begin{proof}
    Note that since $A$ is an abelian ideal in $B$, we have that $M(A)$ acts nilpotently in $B$ which has codimension one. Therefore, there is some $x\in L$ with $x\not\in B$ such that $[A, x]\subset \mathbb{F}x$. Write $[e_i, x] = \lambda_i x$ for a basis $e_2, \ldots, e_{n-1}$ of $A$ and $\lambda_i\in \mathbb{F}$. 
    Assume $\lambda_2\neq0$. Then $A$ is not an ideal of $L$. Denote $Z= \textrm{span}(v_i: 3\leq i\leq n-1)$ where $v_i = e_i-\lambda_2^{-1}\lambda_{j} e_2$.     
    Note that $[v_i, x] =0$. Also, we have  $[e_1, x] = \lambda_2^{-2}[e_1, [e_2, [e_2, x]]] = \mu x$ for some $\mu \in \mathbb{F}$, using the Leibniz rule.    If $I_L\not\subset Z$, then $Z+I_L = A$. But we have $[Z, x] = 0$ and $[x, Z]\subset I_L \subset A$, implying that $A$ is an ideal, which is a contradiction.    
    So assume $I_L\subset Z$.      
    Then, $A'=Z+\mathbb{F}x$ is an abelian subalgebra of codimension two. Certainly, we have $0 = [e_2, [x, x]] = [[e_2, x], x] + [x, [e_2, x]] = 2\lambda_2[x, x]$, so $[x, x] = 0$, and 
    $$[x, v_i] = \lambda_2^{-1}[[e_2, x], v_i] = \lambda_2^{-1}[e_2, [x, v_i]] =- \lambda_2^{-1}[e_2, [v_i, x]]= - [v_i, x]=0.$$
    Moreover $A'$ is an ideal, since $[e_1, x] \in \mathbb{F}x$ and $[x, e_1]\in \mathbb{F}x + I_L \subset A'$. Also, since $[[e_1, e_j], x] = 0$ and $Z$ has codimension one in $A$, we have $[e_1, e_j] \in Z$, implying $[e_1, v_j], [v_j, e_1]\in Z$. The result follows.
\end{proof}

\begin{remark}
     An example of a Leibniz non-Lie algebra of the type of the Lemma~\ref{lemia2} is the following. Consider the vector space $L$ over $\mathbb{F}$ with basis $e_1, e_2, e_3, x$ endowed with the multiplication
    $$[e_1, x] = -[x, e_1] = x, \quad  [e_2, x] = - [x, e_2] = x, \quad  [e_1, e_2] = e_3.$$
    We have $A=\textrm{span}(e_2, e_3)$, $L_{e_2}$ is not nilpotent, $B = A+ \mathbb{F}e_1$, $I_L = \mathbb{F} e_3$ and $A' = \textrm{span}(x, e_3)$.
\end{remark}

\begin{lemma}\label{lemia3}
    Let $L$ be a Leibniz algebra of dimension $n$ over an arbitrary field $\mathbb{F}$ of characteristic $p\neq2$ with $\alpha(L) = n-2$. Let $A$ be an abelian subalgebra of codimension two such that $I_L\subseteq A$. Suppose $B = A + \mathbb{F}e_1$ is a subalgebra of $L$ and suppose $A$ is an ideal of $B$. Moreover, suppose $M(A)$ acts nilpotently in $L$ and suppose $B$ is not an ideal of $L$. Then $A$ is an abelian ideal or $L$ is a Lie algebra.
\end{lemma}
\begin{proof}    
    Suppose $A$ is not an ideal of $L$. Let $e_2, \ldots, e_{n-1}$ be a basis of $A$. Complete the basis of $L$ and write $L=B+\mathbb{F}e_n$. Suppose there is some $a\in A$ such that $L_a(e_n)\not\in B$, then $L_a$ is not nilpotent. Therefore $L_A(L)\subset B$, that is, we have $[A,L]\subset B$. Moreover, since $I_L\subset A$, then we have $[L, A]\subset B$. Since $A$ is not an ideal of $L$, we can assume $[e_2, e_n] = e_1$, using again that $I_L\subset A$. Denote $[e_i, e_n] = \sum_{k=1}^{n-1} \alpha_{ik} e_k$ for $\alpha_{ik}\in \mathbb{F}$ with $3\leq i\leq n-1$. Then if $v_i = e_i - \alpha_{i1}e_2$  we have $[v_i, e_n], [e_n, v_i] \in A$. Thus, we have $$[e_1, v_i] = [[e_2, e_n], v_i] = [e_2, [e_n, v_i]] - [e_n, [e_2, v_i]] = 0, \quad  [v_i, e_1] = [v_i, [e_2, e_n]] = [[v_i, e_2], e_n] + [e_2, [v_i, e_n]] = 0.$$
    Assume without loss of generality that $e_i = v_i$ for $3\leq i\leq n-1$. Let us show that $L$ is anticommutative. 
    
    First, observe that $0 = [e_n, [e_2, e_2]] = [[e_n, e_2], e_2] + [e_2, [e_n, e_2]]$ implies that $[e_1, e_2] = - [e_2, e_1]$.     
   Now, since $B$ is not an ideal of $L$, we have $\mathbb{F}[e_n, e_1] + \mathbb{F}[e_1, e_n]\not\subset B$ and we can assume $[e_1, e_n]\not\in B$. Let us write $[e_1, e_n] = \lambda e_n + b$ for some $\lambda \in \mathbb{F}$ and $b\in B$ with $\lambda\neq0$. Easily, we obtain $[e_1, e_n] = -[e_n, e_1]$. Also, if $[e_1, e_2]=0$, then $$\lambda e_1 + [e_2, b]=\lambda [e_2, e_n] + [e_2, b] = [e_2, [e_1, e_n]] = [e_1, [e_2, e_n]] = [e_1, e_1] \in I_L \subset A,$$ 
   which is a contradiction. Assume $[e_1, e_2]\neq0$, but then $[e_1, [e_1, e_2]] = \mu [e_1, e_2]$ for some $\mu \in \mathbb{F}$. Denote $b=\sum_{i=1}^{n-1} b_i e_i$ and $[e_1, e_2] = \sum_{i=2}^{n-1} \gamma_i e_i$. It follows $\gamma_2=\mu$. By the equation
   $$e_1=[e_2, e_n] = \lambda^{-1}([e_2, [e_1, e_n]] - [e_2, b]) = \lambda^{-1}([[e_2, e_1], e_n] + [e_1, e_1]-b_1 [e_2, e_1]),$$
   we have $\lambda = -\mu$. So the equations $[e_i, [e_1, e_n]] = [[e_i, e_1], e_n] + [e_1, [e_i, e_n]]$ and $[e_1, [e_n, e_i]] = [[e_1, e_n], e_i] + [e_n, [e_1, e_i]]$, imply that $[e_n, e_i] = [e_i, e_n] = 0$ for $3\leq i\leq n-1$. 
   Furthermore, since $[[e_i, e_i], e_1] = 0$ and $[e_2, e_1]\neq 0$, we have that $[e_1, e_1], [e_n, e_n] \in \textrm{span}(e_i: 3\leq i\leq n-1)$. But then, $[e_2, [e_1, e_n]] = [[e_2, e_1], e_n] + [e_1, [e_2, e_n]]$ implies that $[e_1, e_1] = 0$, because $\gamma_2\neq0$.
    Finally, the equation $[e_2, [e_n, e_1]] = [[e_2, e_n], e_1] + [e_n, [e_2, e_1]]$ implies that $[e_n, e_2]=-e_1$ and the equation $0 =[e_1, [e_n, e_n]] = [[e_1, e_n], e_n] + [e_n, [e_1, e_n]] = 2 [e_n, e_n]$ implies that $[e_n, e_n] = 0$. Hence, we conclude that $L$ is a Lie algebra.
\end{proof}

\begin{remark}
    The Lie algebras of Lemma \ref{lemia3} are studied in the case $(2)$ of the proof of \cite[Theorem 4.2]{FNT24}. This algebras are precisely of the form $L = L_1(\gamma) \oplus \mathbb{F}^{n-3}$, where $L_1(\gamma)$ is a simple Lie algebra studied in \cite{A76b}. For this algebras, we have $\beta(L) = n-3$. Note that $\mathfrak{d}(m)\cong L_1(\gamma)$, when $m\in \mathfrak{sl}_2(\mathbb{F})$ is reducible, for some $\gamma$ depending on $m$. Also, if the field has characteristic $p\neq 2$, then $L_1(\gamma)\cong \mathfrak{sl}_2(\mathbb{F})$.
\end{remark}

Given a nilpotent algebra $N$ of dimension $n$, a $s$-dimensional solvable extension of $N$ is a solvable algebra $L$ of dimension $n+s$ having $N$ as its nilradical. For instance, the algebras in the families $\mathfrak{a}, \mathfrak{b}$ are two-dimensional solvable extensions of the two-dimensional abelian algebra. Likewise, the algebras in  $\mathfrak{c}$ are one-dimensional solvable extensions of the Heisenberg algebra

\begin{remark}
    Let us recall the definition of the classical oscillator Lie algebra. The oscillator algebra $\mathfrak{os}(\mathbb{F})$  is the vector space over $\mathbb{F}$ with basis $e_{-1}, e_{0}, e_{1}, \hat{e}_{1}$ together with the bracket given by
$$[e_{-1}, e_{1}] = -[e_{1}, e_{-1}] = \hat{e}_{1}, \quad [e_{-1}, \hat{e}_{1}] = -[\hat{e}_{1}, e_{-1}] = - e_1,  \quad [e_{1}, \hat{e}_{1}] = -[\hat{e}_{1}, e_{1}] = e_{0}.$$
The algebra $\mathfrak{os}(\mathbb{F})$ is a one-dimensional solvable extension of the Heisenberg algebra $\mathfrak{h}(\mathbb{F})$, which is spanned by $e_1, \hat{e}_1, e_0$. 
The algebras such that $\alpha(L)\neq \beta(L)$ in the next lemma are a special type of solvable extensions of the algebra $\mathfrak{h}(\mathbb{F}) \oplus \mathbb{F}^{k}$. For example, if  $\mathbb{F} = \mathbb{R}$, one of this extensions is $\mathfrak{os}(\mathbb{R})$, for which we have $2= 
\alpha(\mathfrak{os}(\mathbb{R})) \neq \beta(\mathfrak{os}(\mathbb{R})) = 1$. 
\end{remark}

We denote by $\mathfrak{e}(\varphi, \vartheta, v, n)$ the one-dimensional solvable extension of the algebra $(\mathfrak{h}(\mathbb{F}) \oplus \mathbb{F}^{n-4}, [\cdot, \cdot]):=H$ with vector space $\mathbb{F}x \oplus H$ and multiplication $[\lambda x + a, \mu x + b]_{\mathfrak{e}} = \lambda \mu v + \lambda \varphi(b) + \mu \vartheta(a) + [a, b]$ for $a, b \in H$ and $\lambda, \mu \in \mathbb{F}$, where $v\in C(H)$ and $\varphi$ is a left derivation of $H$, write $\varphi\in \textrm{Der}_l(H)$, and $\vartheta: H \rightarrow H$ is a linear map. The following result is obtained.

\begin{lemma}\label{lemia1}
    Let $L$ be a Leibniz algebra of dimension $n$ over an arbitrary field $\mathbb{F}$ {of characteristic $p\neq2$} with $\alpha(L) = n-2$. Let $A$ be an abelian subalgebra of codimension two such that $I_L\subseteq A$. Suppose $B = A + \mathbb{F}e_1$ is a ideal of $L$ and suppose $A$ is an ideal of $B$. Then $\beta(L)=n-2$  or     
    $L \cong  \mathfrak{e}(\varphi, \vartheta, v, n)$ for certain $v\in  I_L$ and    
    $\varphi \in \textrm{Der}_l(B)$ such that the induced map $\varphi^* = \vartheta^*:B/C(B)\rightarrow B/C(B)$ is irreducible. In the second case, we have $\textrm{Nil}(L) = B \cong \mathfrak{h}\oplus \mathbb{F}^{n-4}$, $\beta(L)=n-3$ and $C(B)$ is the unique abelian ideal of maximal dimension.
\end{lemma}
\begin{proof}   
    Suppose $A$ is not an ideal of $L$. Assume $M(A)$ acts nilpotently in $L$, otherwise the result follows by Lemma~\ref{lemia2}. Let $e_2, \ldots, e_{n-1}$ be a basis of $A$. Suppose $[A, L]\subset A$, then since $I_L\subset A$, we have $[L, A]\subset A$ and $A$ is an ideal.  So assume $[e_2, e_n]\not\in A$ for some $e_n\in L$ with $e_n\not\in B$. Moreover, since $B$ is an ideal, we can assume $[e_2, e_n] = e_1$. Write $[e_i, e_n] = \sum_{k=1}^{n-1} \alpha_{ik} e_k$ with $\alpha_{ik}\in \mathbb{F}$. Then  for $2\leq i\leq n-1$, we have
    $$[e_i, e_1] = [e_i, [e_2, e_n]] = [e_2, [e_i, e_n]] = \alpha_{i1}[e_2, e_1], \quad \quad [e_1, e_i] = [[e_2, e_n], e_i] = [e_2, [e_n, e_i]] = -\alpha_{i1}[e_2, e_1],$$
    where the second equation uses that $I_L\subset A$. Then, we have $[e_1, e_i] = -[e_i, e_1]$ for $2\leq i\leq n-1$.

    Moreover, denote $v_i = e_i - \alpha_{i1} e_2$ and $Z= \textrm{span}(v_i: 3\leq i\leq n-1)$. Then $[Z, B] = 0$ and $[B, Z] = 0$. 
    Suppose $I_L\not\subset Z$, then $A = Z+I_L$ and we have $[v_i, e_n] = [e_i - \alpha_{i1} e_2, e_n] \in A$ and $[e_n, v_i] \in A$, because $I_L\subset A$.  Hence, $A$  is an abelian ideal.  Similarly, if $[e_1, e_2]\not\in Z$, then $A = Z + \mathbb{F}[e_1, e_2]$ is an ideal. Also, if $[e_1, e_2] = 0$, then $B$ is an abelian subalgebra, since $[e_1, e_1] = [e_1,[e_2, e_n]] = [e_2, [e_1, e_n]] = 0$. The previous assumptions lead to a contradiction, so assume $I_L\subset Z$, $[e_1, e_2]\in Z$ and $[e_1, e_2] \neq 0$.

    Now, we show that $Z$ is an abelian ideal of  $L$. Note that  $[e_1, [v_i, e_n]] = [[e_1, v_i], e_n] + [v_i, [e_1, e_n]] = 0$, implying $\sum_{k=2}^{n-1} \alpha_{ik} [e_1, e_k] = 0$ and $\sum_{k=2}^{n-1} \alpha_{ik} \alpha_{k1} = 0$, so  $\alpha_{i2}=-\sum_{k=3}^{n-1} \alpha_{ik} \alpha_{k1}$. Then $[v_j, e_n] = \sum_{k=2}^{n-1} \alpha_{ik} e_k = \sum_{k=3}^{n-1} \alpha_{ik}v_j$, concluding $[v_j, e_n] \in Z$ and $[e_n, v_j]\in Z$, since $I_L\subset Z$. 
    Next, by the equations 
    $$[e_1, e_1] = [e_1, [e_n, e_2]] = [[e_1, e_n], e_2] + [e_n, [e_1, e_2]],$$
    $$[e_n, [e_2, e_1]] = [[e_n, e_2], e_1] + [e_2, [e_n, e_1]] = [e_1, e_1] + [e_2, [e_n, e_1]],$$
    we have $2[e_1, e_1] =  [[e_1, e_n], e_2] - [e_2, [e_n, e_1]] = 0$, so $[e_1, e_1] = 0$.  
     
     Hence, $B$ is a Lie algebra. In fact, $B = \mathfrak{h}(\mathbb{F}) \oplus \mathbb{F}^{n-4}$, where $\mathfrak{h}(\mathbb{F})$ is the Heisenberg algebra.     
     Now, consider $A' = Z + \mathbb{F}e_1$. The space $A'$ is an abelian subalgebra of codimension two. Moreover, we have $I_L\subset A'$ and $A'$ is an ideal of $B$.     
     By Lemma \ref{lemia2} we can assume $A'$ acts nilpotently, otherwise we have $\beta(L) = n-2$. Therefore, we have that $M(B)$ acts nilpotently in $L$. If $L_{e_n}$ is nilpotent, then $L$ is nilpotent by the Engel's theorem for Leibniz algebras, see \cite[Theorem 2.1, p.44]{AOR20}. In that case, we have $\beta(L) = n-2$ by \cite[Proposition 3.2]{CT23}.
     So suppose $L_{e_n}$ is not nilpotent. 

     Furthermore, consider the induced map $L_{e_n}^{*}:B/Z\rightarrow B/Z$. If $v\in B/Z$ is an eigenvector of $L_{e_n}^{*}$, then there is some $v'\in B$ such that $L_{e_n}(v') \in Z + \mathbb{F}v'$ and we have that $Z + \mathbb{F}v'$ is an abelian ideal of codimension two. Therefore, assume $L_{e_n}^{*}$ is irreducible in $B/Z$.
     Now, suppose $A'$ is an abelian ideal of maximal dimension of $L$. Then we can easily see that $A'\subset B$ and then $Z\subset A'$. By this fact and $L_{e_n}^{*}$ being irreducible in $B/Z$, we conclude that $Z$ is an abelian ideal of maximal dimension, i.e. we have $\beta(L)=n-3$. The statement in the lemma follows.
\end{proof}

\begin{remark}
    { 
    An example of a Leibniz non-Lie algebra of the type $\alpha(L)\neq \beta(L)$ of the Lemma~\ref{lemia1} is the following. Consider the vector space $L$ over $\mathbb{R}$ with basis $e_1, e_2, e_3, e_4$ endowed with the multiplication
    $$[e_1, e_2] = -[e_2, e_1] = e_3, \quad  [e_1, e_4] = - [e_4, e_1] = -e_2, \quad  [e_2, e_4]=-[e_4, e_2] = e_1, \quad [e_4, e_4] = e_3.$$
    We have $A=\textrm{span}(e_2, e_3)$, $L_{e_4}$ is not nilpotent, $B = A+ \mathbb{F}e_1$ is the nilradical of $L$, $I_L = \mathbb{F} e_3$ and $L\cong \mathfrak{e}(\varphi, v, n)$, where $v=e_3$ and $\varphi: B \rightarrow B$ is given by $\varphi(e_1) = e_2, \varphi(e_2) = -e_1$ and $\varphi(e_3) = 0$. Note that the map $\varphi^*$ is irreducible.    
    }
\end{remark}

The previous lemmas together with the previous section already characterize the Lie algebras with an abelian subalgebra of codimension two. 
In the next theorem, we complete this description for Leibniz non-Lie algebras.

\begin{theorem}
    Let $L$ be a Leibniz non-Lie algebra of dimension $n$ over an arbitrary field $\mathbb{F}$ of characteristic $p\neq2$ with $\alpha(L) = n-2$. Let $A$ be an abelian subalgebra of codimension two. Suppose there is a subalgebra $B$ of $L$ of codimension one containing $A$. Then $\beta(L)=n-2$ or $L$ is $3$-step solvable and     
    $L \cong  \mathfrak{e}(\varphi, \vartheta, v, n)$ for certain $v\in  I_L$ and    
    $\varphi \in \textrm{Der}_l(B)$ such that the induced map $\varphi^* = \vartheta^*:B/C(B)\rightarrow B/C(B)$ is irreducible, where $B=\textrm{Nil}(L)\cong \mathfrak{h}\oplus \mathbb{F}^{n-4}$ (see Lemma~\ref{lemia1}).  
\end{theorem}
\begin{proof}
By \cite[Theorem 2.1]{CT23}, we can assume $A$ is an abelian ideal of dimension $n-2$ of $B$ as $p\neq2$.  Let $e_2, \ldots, e_{n-1}$ be a basis of $A$. Write $B=A+\mathbb{F}e_1$ and $L = B+\mathbb{F}e_n$. It follows that $B\subseteq N(A)$. If $N(A)=L$, then $A$ is an abelian ideal of $L$ and we have $\beta(L) = n-2$. So assume $N(A)=B$.     
Let us distinguish three situations depending on $I_L$. The case $I_L\subseteq A$ is already studied in Lemma~\ref{lemia2}, Lemma~\ref{lemia3} and Lemma~\ref{lemia1}, obtaining the respective claim in the statement. Assume now that $I_L\not\subset A$. If $ L=A+ I_L$, then $L$ is clearly solvable. 

So suppose that $L\neq A+I_L$. If $A$ acts nilpotently on $L$, there exists $k\geq 0$ such that $L_A^k(L)\not \subseteq A+I_L$, but $L_A^{k+1}(L)\subseteq A+I_L$. Let $x\in L_A^k(L)\setminus (A+I_L)$, so $[A,x]\subseteq A+I_L$. Then $L=A+I_L+\mathbb{F}x$, $L^2\subseteq A+I_L$ and $L$ is solvable. 
Suppose now that $A$ does not act nilpotently on $L$. Let $L=B\oplus L_1$ be the Fitting decomposition of $L$ relative to $M(A)$, and put $L_1=\mathbb{F}x$. If $B=A\oplus I_L$, we have $L^2\subseteq I_L+\mathbb{F}x$, $L^{(2)}\subseteq I_L$ and $L$ is solvable again. If $B\neq A+I_L$, then $L=B+I_L$, in which case $L^2\subseteq A+I_L$ and $L$ is solvable once more.

Finally, suppose that $\beta(L)\neq n-2$. Then $L$ must be as in \cite[Theorem 3.5 (iii)]{CT23}. Let $N$ be the nilradical of $L$. Then $A+I_L\subseteq N$ and $I_L \subseteq C(N)$, whence $A+I_L$ is abelian, a contradiction. The result follows.    
\end{proof}

\begin{corollary}
    Let $L$ be a Leibniz algebra of dimension $n$ over a quadratically closed field $\mathbb{F}$ {of characteristic $p\neq2$} with $\alpha(L) = n-2$. Then $\beta(L)=n-2$ or $L$ is isomorphic to the Lie algebra $\mathfrak{sl}_2(\mathbb{F}) \oplus \mathbb{F}^{n-3}$. 
\end{corollary}

\begin{proposition}
    Let $L$ be a Leibniz algebra with an abelian ideal of codimension two. Then $L$ is solvable.
\end{proposition}
\begin{proof}
    Let $B$ be an abelian ideal of codimension two. If $I_L \subset B$, write $L = B + \mathbb{F}x + \mathbb{F}y$. Then $L^2 \subseteq B + \mathbb{F}[x, y]$ and $L^{(2)} \subseteq B$ and $L^{(3)} = 0$.    
    If $L, B\neq I_L+B$. Then $L = \mathbb{F}x + I_L + B$  and we have $L^2 \subseteq I_L + B$, $L^{(2)} \subseteq B$ and $L^{(3)} = 0$. Finally, if $L = I_L+ B$, then $L^{(2)} = 0$.
\end{proof}

Combining the results from section~\ref{sec2} and section~\ref{sec3}, we have proven Theorem~\ref{mainthm}.

\end{document}